\documentclass[12pt]{article}

\usepackage[T1]{fontenc}
\usepackage[english]{babel}
\usepackage{amsmath,amssymb,amsthm}
\usepackage{newtxtext,newtxmath}
\usepackage{bm}
\usepackage{cite}
\usepackage{microtype}
\usepackage[margin=1in]{geometry}
\usepackage{float}
\usepackage{booktabs}
\usepackage{xcolor}
\usepackage{tikz}
\usetikzlibrary{calc,arrows.meta,positioning,decorations.pathreplacing,fit}
\usepackage{indentfirst}
\usepackage[colorlinks=true,linkcolor=blue,citecolor=red]{hyperref}
\hypersetup{
	pdftitle={Characteristic Polynomials of Graph- and Digraph-Deleted Catalan Arrangements},
	pdfauthor={Yanru Chen, Ang Li, and Suijie Wang},
	pdfsubject={Hyperplane arrangements and characteristic polynomials},
	pdfkeywords={Catalan arrangements, finite field method, graphical Stirling numbers, directed path covers, gain graphs}
}
\allowdisplaybreaks[4]

\definecolor{figblue}{RGB}{45,92,150}
\definecolor{figgreen}{RGB}{72,135,94}
\definecolor{figgold}{RGB}{176,125,38}
\definecolor{figred}{RGB}{166,72,58}
\definecolor{figgray}{RGB}{90,90,90}
\tikzset{
	fig arrow/.style={-{Latex[length=2.2mm]}, thick, figgray},
	fig block/.style={draw=figblue, very thick, rounded corners=2pt,
		align=center, font=\small, minimum height=6mm, fill=figblue!12},
	fig gap/.style={draw=figgold, thick, rounded corners=2pt,
		align=center, font=\small, minimum height=6mm, fill=figgold!18},
	fig vertex/.style={circle, draw=figgray, thick, fill=white,
		inner sep=1.2pt, minimum size=5.5mm, font=\small}
}

\newtheorem{theorem}{Theorem}[section]
\newtheorem{maintheorem}{Theorem}

\newtheorem{lemma}[theorem]{Lemma}
\newtheorem{prop}[theorem]{Proposition}
\newtheorem{coro}[theorem]{Corollary}
\theoremstyle{definition}
\newtheorem{defi}[theorem]{Definition}
\newtheorem{example}[theorem]{Example}
\theoremstyle{remark}
\newtheorem*{remark*}{Remark}

\newcommand{\A}{\mathcal A}
\newcommand{\C}{\mathcal C}
\newcommand{\Ish}{\operatorname{Ish}}
\newcommand{\Stir}{\operatorname{Stir}}
\newcommand{\pathnum}{\operatorname{path}}
\newcommand{\pc}{\operatorname{pc}}
\newcommand{\Path}{\operatorname{Path}}

\numberwithin{equation}{section}

\begin{document}
\enlargethispage{1.7\baselineskip}
\vspace*{-0.55em}
\begin{center}
	{\LARGE\bfseries
		Characteristic Polynomials of Graph- and\\[4pt]
		Digraph-Deleted Catalan Arrangements\par}
	\vspace{0.75em}
	Yanru Chen\(^{1}\)\quad Ang Li\(^{2}\)\quad Suijie Wang\(^{3}\)\\[4pt]
	\(^{1,2,3}\)School of Mathematics, Hunan University\\
	Changsha 410082, Hunan, P. R. China\\[3pt]
	{\footnotesize Correspondence:
	\(^{1}\)yanruchen@hnu.edu.cn, \(^{2}\)leonli@hnu.edu.cn, \(^{3}\)wangsuijie@hnu.edu.cn}
\end{center}
\vspace{0.15em}
\begin{center}
\begin{minipage}{0.90\textwidth}
\small
	\noindent\textbf{Abstract.}
	We develop a finite-field stratification for characteristic polynomials of
	deletion subarrangements of the full \(m\)-Catalan arrangement.  It reduces the
	count to cyclic placements of rigid blocks and yields falling-factorial
	expansions for deletions indexed by graphs, digraphs, and gain-labeled
	digraphs.  The coefficients are graphical Stirling numbers for zero-layer
	deletions, directed matching numbers when the deleted layer \(\ell\) satisfies
	\(1\le \ell\le\lfloor m/2\rfloor\), directed path-cover numbers when
	\(\lfloor m/2\rfloor<\ell\le m\), and admissible gain-labeled arc sets for
	multilayer deletions.  For \(\ell=m\), a complementary path-cover expansion
	yields factorization consequences.  The method also gives formulas for
	directed Ish-type arrangements in terms of path-cycle covers and outdegrees.
	
	\vspace{0.35em}
	
	\vspace{0.15em}
\end{minipage}
\end{center}

\section{Introduction}
Hyperplane arrangements whose hyperplanes have the form \(x_i-x_j=s\) are
integral deformations of the braid arrangement; for general background on
hyperplane arrangements, see~\cite{orlikTerao1992,stanley2007}.  Classical
examples include the Catalan, Shi, semiorder, and Ish arrangements
\cite{shi1986,stanley1996,armstrong2012}.  Postnikov and Stanley developed a
general framework for such Coxeter-arrangement deformations, and Bernardi later
gave tree formulas for regions and for characteristic and coboundary polynomials
of integer braid deformations
\cite{postnikov2000,bernardi2018}.  From the gain-graph and affinographic
viewpoints, the characteristic polynomial over a finite field counts vertex
labelings avoiding prescribed differences
\cite{zaslavsky1989,forge2007,berthome2009}.

We study three classes of deletion subarrangements of the full \(m\)-Catalan
arrangement: zero-layer deletions indexed by a graph, single translated-layer
deletions indexed by a digraph, and multilayer deletions indexed by
gain-labeled digraphs.  Cyclic block-and-gap enumeration produces coefficients
with direct graph-theoretic interpretations.

Throughout, \(n\) is a positive integer and \([n]=\{1,2,\dots,n\}\).
	We use the following graph- and digraph-deleted Catalan arrangements.
	Let
	\[
	\mathcal B_n
	=
	\{x_i-x_j=0\mid 1\le i<j\le n\}
	\]
	be the braid arrangement.  For an integer \(m\ge0\), let
	\[
	\mathcal C_n^m
	=
	\mathcal B_n
	\cup
	\{x_i-x_j=s\mid 1\le i\ne j\le n,\ 1\le s\le m\}
	\]
	be the full \(m\)-Catalan arrangement.  For a simple graph
	\(G=([n],E(G))\), define the graph-deleted Catalan arrangement
	\[
	\C_n^m(G)
	=
	\mathcal C_n^m
	\setminus
	\{x_i-x_j=0\mid \{i,j\}\in E(G)\}.
	\]
	For a simple digraph \(\Gamma=([n],A(\Gamma))\) and
	\(1\le \ell\le m\), define the \(\ell\)-layer digraph-deleted Catalan
	arrangement
	\[
	\C_n^{m,\ell}(\Gamma)
	=
	\mathcal C_n^m
	\setminus
	\{x_i-x_j=\ell\mid (i,j)\in A(\Gamma)\}.
	\]
	For \(m\ge1\), more generally, for \(S\subseteq[m]\) and simple digraphs
	\(\Gamma_\bullet=(\Gamma_s)_{s\in S}\), with
	\(\Gamma_s=([n],A_s)\), define
	\[
	\C_n^{m,S}(\Gamma_\bullet)
	=
	\mathcal C_n^m
	\setminus
	\{x_i-x_j=s\mid s\in S,\ (i,j)\in A_s\}.
	\]

For a hyperplane arrangement \(\A\), let \(L(\A)\) be the intersection poset
of nonempty intersections of subfamilies of \(\A\), ordered by reverse
inclusion, with minimum element \(\hat 0\) equal to the ambient space.  If
\(\mu\) denotes the M\"obius function of \(L(\A)\), we write
\[
\chi(\A,t)=\sum_{X\in L(\A)}\mu(\hat 0,X)t^{\dim X}
\]
for the characteristic polynomial of \(\A\).

Throughout the paper, \(G\) is a simple graph and \(\bar G\) denotes its usual
complement on the same vertex set.  We write \(\emptyset_n\) for the empty graph
on vertex set \([n]\).  Unless stated otherwise, a digraph is simple and
loopless.
For a simple digraph \(\Gamma\), the notation \(\bar\Gamma\) means the
complement in the complete loopless digraph on the same vertex set.  In the
directed Ish section, \(\bar\Gamma(1)\) denotes the complement inside the
explicitly displayed indexing set, not inside the complete loopless digraph.
We use the falling factorial notation
\[
x^{\underline r}=x(x-1)\cdots(x-r+1),\qquad x^{\underline 0}=1.
\]

For the zero-layer formula, let \(\Stir_k(G)\) be the graphical Stirling number
defined in
\cite{duncan2009}, i.e., the number of partitions of the vertex set of \(G\)
into \(k\) nonempty independent sets, where an \emph{independent set} is a set
of vertices no two of which are adjacent.  Equivalently, \(k!\Stir_k(G)\) is the
number of proper \(k\)-colorings of \(G\) that use all \(k\) colors.

\begin{maintheorem}[Graph-deleted Catalan formula]\label{main-1}
	Let \(m\ge1\), and let \(G=([n],E(G))\) be a simple graph.  The arrangement
	\(\C_n^m(G)\) has characteristic polynomial
	\[
	\chi(\C_n^m(G),t)
	=
	t\sum_{k=1}^{n}
	\Stir_k(\bar G)(t-mk-1)^{\underline{k-1}}.
	\]
\end{maintheorem}
The directed layer deletions considered here are related to several earlier
directions, but the role of the digraph is different in each case.
Athanasiadis used the finite field method to obtain product formulas for
structured type-\(A\) deformations satisfying order or closure conditions
\cite[Theorems~3.4 and~3.9]{athanasiadis1996}.  Berthom{\'e} \textit{et al.}
developed gain-graph chromatic reductions that apply, among other examples, to
the Shi, Linial, and Catalan gain graphs~\cite{berthome2009}.  Armstrong and
Rhoades studied deleted Shi and Ish arrangements indexed by arbitrary
subgraphs of the complete graph and showed that the Shi--Ish symmetry preserves
several invariants, including the characteristic polynomial
\cite{armstrong2012}.  Abe, Tran, and Tsujie later introduced
vertex-weighted digraph arrangements that generalize Catalan, Shi, Ish, and
intermediate Shi--Ish arrangements, with a focus on freeness and operations
preserving characteristic polynomials~\cite{abeTranTsujie2024}.
In contrast, Theorem~\ref{main-2} deletes a single translated layer from the
full \(m\)-Catalan arrangement, with arbitrary simple digraph data.  Its
coefficients are directed matching numbers or directed path-cover numbers,
depending on the deleted layer.

We use standard digraph terminology~\cite{Bang09}.  A \emph{matching} in
\(\Gamma\) is a set of pairwise vertex-disjoint arcs; let
\(\operatorname{mat}_k(\Gamma)\) be the number of \(k\)-arc matchings.  A
\emph{path cover} of \(\Gamma\) is a set of vertex-disjoint directed paths
covering all vertices, with isolated vertices allowed as paths of length \(0\);
let \(\pathnum_k(\Gamma)\) be the number of path covers with exactly \(k\)
directed paths. For readability in the theorem, set
\[
\theta_n^{m,\ell}(k)=(m-\ell)k+\ell n+1.
\]

\begin{maintheorem}[Digraph-deleted Catalan formula]\label{main-2}
	Let \(m\ge1\), let \(1\le \ell\le m\), and let
	\(\Gamma=([n],A(\Gamma))\) be a simple digraph.  Then
	\begin{itemize}
		\item[(1)] When \(1\le \ell\le \lfloor m/2\rfloor\), the arrangement
		\(\C_n^{m,\ell}(\Gamma)\) has characteristic polynomial
		\[
		\chi(\C_n^{m,\ell}(\Gamma),t)
		=
		t\sum_{k=\lceil n/2\rceil}^{n}
		\operatorname{mat}_{n-k}(\Gamma)
		(t-\theta_n^{m,\ell}(k))^{\underline{k-1}}.
		\]
		\item[(2)] When \(\lfloor m/2\rfloor+1\le \ell\le m\), the arrangement
		\(\C_n^{m,\ell}(\Gamma)\) has characteristic polynomial
		\[
		\chi(\C_n^{m,\ell}(\Gamma),t)
		=
		t\sum_{k=1}^{n}
		\pathnum_k(\Gamma)
		(t-\theta_n^{m,\ell}(k))^{\underline{k-1}}.
		\]
	\end{itemize}
\end{maintheorem}
When \(\ell=m\), the formula above also gives a simple criterion:
integer linear factorization is controlled by the corresponding path-cover
polynomial and, for fixed \(\Gamma\), is independent of \(m\).  Changing \(m\)
only translates the nonzero roots.  This criterion is stated and proved in
Section~\ref{sec-digraph-deleted}.

The digraph formula above extends to deletions in several translated layers;
this is the third main result of the paper.  When \(S=\{\ell\}\), it recovers
Theorem~\ref{main-2}.  For several deleted layers, however, the coefficients
are no longer ordinary matching or path-cover numbers.  They are instead
weighted counts of admissible sets of
deleted gain-labeled arcs, which record compatible collections of deleted
translated hyperplanes. 
The precise definition of admissible sets of deleted gain-labeled arcs, together
with the quantities \(c(F)\) and \(W(F)\), is given in
Section~\ref{sec-simultaneous-translated-deletions}.

\begin{maintheorem}[Simultaneous deletions in translated layers]
\label{thm:simultaneous-translated-deletions}
	Let \(m\ge1\), let \(S\subseteq[m]\), and for each \(s\in S\) let
	\(\Gamma_s=([n],A_s)\) be a simple digraph.  Put
	\(\Gamma_\bullet=(\Gamma_s)_{s\in S}\).  Then
	\[
	\chi(\C_n^{m,S}(\Gamma_\bullet),t)
	=
	t\sum_F
	\bigl(t-W(F)-m\,c(F)-1\bigr)^{\underline{c(F)-1}},
	\]
	where the sum is over all admissible sets \(F\) of deleted gain-labeled arcs.
\end{maintheorem}

Table~\ref{tab:main-results} summarizes the main characteristic-polynomial
formulas, the arrangements to which they apply, and the corresponding
coefficient enumerators.

\begin{table}[H]
	\centering
	\small
	\renewcommand{\arraystretch}{1.12}
	\setlength{\tabcolsep}{2.5pt}
	\begin{tabular*}{\textwidth}{@{}l@{\extracolsep{\fill}}ll@{}}
		\toprule
		Result & Arrangement & Coefficients or consequence\\
		\midrule
		Theorem~\ref{main-1} &
		\(\C_n^m(G)\) &
		graphical Stirling numbers of \(\bar G\)\\
		Theorem~\ref{main-2}(1) &
		\(\C_n^{m,\ell}(\Gamma)\) &
		directed matching numbers, \(1\le \ell\le\lfloor m/2\rfloor\)\\
		Theorem~\ref{main-2}(2) &
		\(\C_n^{m,\ell}(\Gamma)\) &
		directed path-cover numbers, \(\lfloor m/2\rfloor<\ell\le m\)\\
		Theorem~\ref{thm:simultaneous-translated-deletions} &
		\(\C_n^{m,S}(\Gamma_\bullet)\) &
		admissible sets of deleted gain-labeled arcs\\
		Theorem~\ref{thm:digraph-catalan-l-equals-m} &
	\(\C_n^{m,m}(\Gamma)\) &
		inclusion--exclusion over path covers of \(\bar\Gamma\)\\
		Corollary~\ref{cor:factorization-criterion} &
		\(\C_n^{m,m}(\Gamma)\) &
		factorization via path-cover polynomial\\
		\bottomrule
	\end{tabular*}
	\caption{A summary of the main characteristic-polynomial formulas.}
	\label{tab:main-results}
\end{table}

All three deletion formulas are proved by the same finite-field strategy:
one translates the deleted hyperplanes into forbidden differences over
\(\mathbb F_p\), decomposes the resulting configurations into cyclic block
placements, and reads off the corresponding falling-factorial coefficients.

Section~\ref{sec-deleted-catalan-arrangements} contains the finite-field setup
and proves the deleted-Catalan formulas.  Section~\ref{sec-path-cover-consequences}
derives path-cover consequences, including the Lah inversion and factorization
formulas.  Section~\ref{sec-ish-type-applications} treats Ish-type
applications.

\section{Deleted-Catalan arrangements}\label{sec-deleted-catalan-arrangements}

This section contains the finite-field proof of the three deleted-Catalan
formulas.  We first recall the placement model used throughout the paper, and
then treat zero-layer, single-layer, and multilayer deletions in increasing
generality.

\subsection{Finite-field preliminaries}\label{sec-prelim}
Let \(\A\) be a hyperplane arrangement in \(\mathbb F^n\).  Its complement is
\[
	M(\mathcal{A}) = \mathbb F^n \setminus \bigcup_{H\in \mathcal{A}} H.
\]
For \(X\in L(\A)\), the \emph{restriction} of \(\A\) to \(X\) is
\[
	\A^X=\{X\cap H\mid H\in\A,\ X\cap H\ne X,\ X\cap H\ne\emptyset\},
\]
viewed as an arrangement in \(X\), and \(M(\A^X)\) is taken inside \(X\).  The
standard restriction-stratum decomposition is
\begin{equation}\label{decomposition}
\mathbb F^n=\bigsqcup_{X\in L(\mathcal A)} M(\mathcal A^X).
\end{equation}

An affine hyperplane \(H\subseteq\mathbb R^n\) is \emph{integral} if it has a
defining equation
\[
H:\quad a_1x_1+\cdots+a_nx_n=c,
\]
with \(a_1,\dots,a_n,c\in\mathbb Z\).  An arrangement is \emph{integral} if all
its hyperplanes are integral.  For an integral arrangement \(\A\), let
\(\A_p\) be the arrangement over \(\mathbb F_p\) obtained by reducing all
defining equations modulo \(p\), for all sufficiently large primes \(p\).  Set
\[
M_p(\A)=M(\A_p)
\]
for the finite-field complement.

\begin{prop}[Finite field method {\cite{crapo1970,athanasiadis1996,Bjorner-Ekedahl1997}}]
	\label{prop:finite-field-method}
	Let \(\mathcal A\) be an integral arrangement in \(\mathbb R^n\).  For all
	sufficiently large primes \(p\), we have
	\[
	\chi(\mathcal A,p)
	=
	\#M_p(\mathcal{A}).
	\]
\end{prop}

\begin{lemma}\label{lem:restriction-strata}
	Let \(\mathcal A\) be an arrangement and let \(X\in L(\mathcal A)\).  For
	\(y\in X\), we have \(y\in M(\mathcal A^X)\) if and only if, for every
	\(H\in\mathcal A\),
	\[
	y\in H
	\quad\Longleftrightarrow\quad
	X\subseteq H.
	\]
	Thus the stratum \(M(\mathcal A^X)\) records exactly the hyperplanes of
	\(\mathcal A\) that contain \(X\).
\end{lemma}

\begin{proof}
	If \(X\subseteq H\), then every point of \(X\) lies in \(H\).  If
	\(X\not\subseteq H\), then either \(X\cap H=\emptyset\), in which case no point
	of \(X\) lies in \(H\), or \(X\cap H\) is a hyperplane of the restriction
	\(\mathcal A^X\).  The condition \(y\in M(\mathcal A^X)\) excludes the
	latter possibility.
\end{proof}

We identify \(\mathbb F_p=\{0,1,\ldots,p-1\}\) and represent its elements by
\(p\) boxes arranged cyclically and labeled by \(0,1,\ldots,p-1\).  A point
\((x_1,\ldots,x_n)\in\mathbb F_p^n\) is represented by placing \(i\) in the
box labeled by \(x_i\), for each \(i\in[n]\).  By the finite field method,
\(\chi(\mathcal A,p)\) counts those
placements satisfying the avoidance restrictions imposed by the hyperplanes.
In particular, if \(x_i-x_j=c\) belongs to \(\mathcal A\) with \(c>0\), then
\(i\) may not be placed \(c\) boxes clockwise after \(j\), where each box is
counted whether it is occupied by some element of \([n]\) or empty.  This placement model, together
with the restriction decomposition \eqref{decomposition}, is the main tool in
the proofs below.

\begin{lemma}\label{lem:cyclic-gap-count}
	Let \(k\ge1\) distinguishable blocks have total length \(N\le p\), and place
	them on a cycle of \(p\) boxes labeled by \(0,1,\ldots,p-1\), with arbitrary
	cyclic order.  If the remaining
	\(p-N\) boxes are distributed freely among the \(k\) cyclic gaps, then the
	number of placements is
	\[
	p(k-1)!\binom{p-N+k-1}{k-1}
	=
	p(p-N+k-1)^{\underline{k-1}}.
	\]
\end{lemma}

\begin{proof}
	Choose the starting position of one distinguished block; this gives \(p\)
	choices.  The other blocks have \((k-1)!\) cyclic orders.  Finally, the
	remaining \(p-N\) boxes form a weak composition into \(k\) gaps.  The displayed
	identity follows from
	\((k-1)!\binom{a}{k-1}=a^{\underline{k-1}}\).
\end{proof}

Figure~\ref{fig:cyclic-gap} records the cyclic placement model used repeatedly in
the finite-field counts below.  The blocks \(R_i\) have total length
\(N\), and only the dashed cyclic gaps are freely distributed.

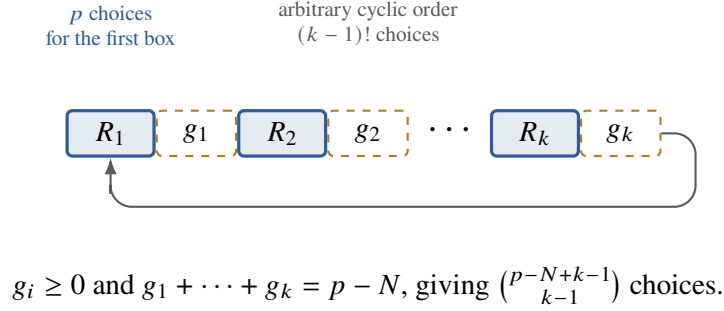
\begin{figure}[H]
	\centering
	\begin{tikzpicture}[node distance=0pt]
		\node[fig block, minimum width=1.15cm] (b1) {\(R_1\)};
		\node[fig gap, dashed, fill=white, minimum width=1.05cm, right=of b1] (g1) {\(g_1\)};
		\node[fig block, minimum width=1.15cm, right=of g1] (b2) {\(R_2\)};
		\node[fig gap, dashed, fill=white, minimum width=1.05cm, right=of b2] (g2) {\(g_2\)};
		\node[font=\large, right=1mm of g2] (dots) {\(\cdots\)};
		\node[fig block, minimum width=1.15cm, right=1mm of dots] (bk) {\(R_k\)};
		\node[fig gap, dashed, fill=white, minimum width=1.05cm, right=of bk] (gk) {\(g_k\)};
		\node[font=\scriptsize, figblue, align=center, above=7mm of b1]
		{\(p\) choices\\for the first box};
		\node[font=\scriptsize, figgray, align=center, above=7mm of g2]
		{arbitrary cyclic order\\\((k-1)!\) choices};
		\draw[fig arrow, rounded corners=8pt]
		(gk.east) -- ++(0.45,0) |- ($(b1.south)+(0,-0.65)$) -| (b1.south);
		\node[font=\small, align=center, below=13mm of g2]
		{\(g_i\ge0\) and \(g_1+\cdots+g_k=p-N\), giving
		\(\binom{p-N+k-1}{k-1}\) choices.};
	\end{tikzpicture}
	\caption{The cyclic placement model in Lemma~\ref{lem:cyclic-gap-count}.}
	\label{fig:cyclic-gap}
\end{figure}

\subsection{Zero-layer deletions}\label{sec-graph-deleted}

Here the deleted zero-layer hyperplanes are indexed by a simple graph, while
all translated layers \(x_i-x_j=s\), \(1\le s\le m\), remain present.

\begin{proof}[Proof of Theorem~\ref{main-1}]	
	Let \(p\) be a sufficiently large prime, chosen in particular so that
	\(p>n(m+1)\).
	
	For \(y=(y_1,\ldots,y_n)\in M_p(\C_n^m(G))\), let \(\pi(y)\) be the
	partition of \([n]\) into equality classes of the coordinates:
	\(i\) and \(j\) lie in the same part if and only if \(y_i=y_j\).
	If \(\{i,j\}\in E(\bar G)\), then the zero-layer hyperplane \(x_i=x_j\)
	remains in \(\C_n^m(G)\).  Hence no part of \(\pi(y)\) contains an edge of
	\(\bar G\), and every part is an independent set of \(\bar G\).
	
	Conversely, fix a partition \(\pi=\{B_1,\ldots,B_k\}\) of \([n]\) into
	independent sets of \(\bar G\).  Points with equality partition \(\pi\) are
	obtained by placing all elements in each part \(B_s\) in one box, with
	different parts in different boxes.  Since all translated Catalan
	hyperplanes \(x_i-x_j=s\), \(1\le i\ne j\le n\), \(1\le s\le m\), remain
	present, every cyclic gap between consecutive occupied boxes must contain at
	least \(m\) empty boxes.  Thus each occupied box, together with the \(m\)
	mandatory empty boxes immediately after it, forms a placement block of length
	\(m+1\).  The \(k\) placement blocks have total mandatory length
	\(N=(m+1)k\), and the remaining boxes are distributed freely among the
	\(k\) cyclic gaps.  Lemma~\ref{lem:cyclic-gap-count} gives
	\[
	\#\{y\in M_p(\C_n^m(G)):\pi(y)=\pi\}
	=
	p\,(p-mk-1)^{\underline{k-1}}.
	\]
	There are \(\Stir_k(\bar G)\) such partitions with \(k\) parts.  Therefore
	\[
	\# M_p(\C_n^m(G))
	=
	p\sum_{k=1}^{n}\Stir_k(\bar G)(p-mk-1)^{\underline{k-1}}.
	\]
By the finite field method, the left-hand side is
\(\chi(\C_n^m(G),p)\) for all sufficiently large primes \(p\).  Hence the
displayed identity holds for infinitely many values of \(p\), and therefore
as a polynomial identity after replacing \(p\) by \(t\).
\end{proof}

We record several specializations of Theorem~\ref{main-1}.  When
\(G=\emptyset_n\), the arrangement
\(\C_n^m(\emptyset_n)\) is the full \(m\)-Catalan arrangement.  In this case
\(\overline{\emptyset_n}=K_n\), so \(\Stir_k(K_n)=0\) unless \(k=n\) and
\(\Stir_n(K_n)=1\).  At the other extreme, when \(G=K_n\), the arrangement
\(\C_n^m(K_n)\) is the \(m\)-semiorder arrangement.  Since
\(\bar K_n=\emptyset_n\), \(\Stir_k(\emptyset_n)\) equals the Stirling number of the second
kind \(\left\{\begin{smallmatrix}n\\ k\end{smallmatrix}\right\}\).  Then we have
\begin{align*}
	\chi(\C_n^m(\emptyset_n),t) = t(t-mn-1)^{\underline{n-1}},
	\quad
	\chi(\C_n^m(K_n),t) = t\sum_{k=1}^{n}\left\{\begin{matrix}n\\ k\end{matrix}\right\}(t-mk-1)^{\underline{k-1}}.
\end{align*}
Our formula applies directly when graphical Stirling numbers are known.  Galvin
and Thanh~\cite{galvin2013} computed \(\Stir_k(G)\) for forests and cycles.
Substituting their results into Theorem~\ref{main-1}, we obtain: if
\(\bar G=F_{n,c}\) is a forest on \(n\) vertices with \(c\) connected
components,
\[
\chi(\C_n^m(G),t)
=
t\sum_{k=1}^{n}
\left(
\sum_{i=0}^{c-1}
\binom{c-1}{i}
\left\{\begin{matrix}n-1-i\\ k-1\end{matrix}\right\}
\right)
(t-mk-1)^{\underline{k-1}}.
\]
If \(\bar G=C_n\) is a cycle with \(n\geq 2\), then
\[
\chi(\C_n^m(G),t)
=
t\sum_{k=1}^{n}
\left(
\sum_{i=0}^{n-2}
(-1)^i
\left\{\begin{matrix}n-1-i\\ k-1\end{matrix}\right\}
\right)
(t-mk-1)^{\underline{k-1}}.
\]

\subsection{Single-layer deletions}\label{sec-digraph-deleted}

Here the zero-layer hyperplanes are complete, while deleted hyperplanes in one
translated layer \(x_i-x_j=\ell\) are indexed by a simple digraph.  The count
depends on whether the selected deleted arcs can form long directed paths.

\begin{proof}[Proof of Theorem~\ref{main-2}]
	Let \(p\) be a sufficiently large prime, chosen in particular so that
	\(p>n(m+1)\), and let
	\[
	\A(\Gamma,\ell)
	=
	\{H_{ij}^{\ell}:x_i-x_j=\ell,\ (i,j)\in A(\Gamma)\}
	\]
	be the arrangement in \(\mathbb F_p^n\) consisting of the deleted \(\ell\)-th layer
	hyperplanes.  From the decomposition \eqref{decomposition}, we have
	\begin{align}
		M_p(\C_n^{m,\ell}(\Gamma))
		&=
		\bigsqcup_{X\in L(\A(\Gamma,\ell))}
		\bigl(M_p((\A(\Gamma,\ell))^X)
		\cap M_p(\C_n^{m,\ell}(\Gamma))\bigr).
		\label{eq:deleted-layer-decomposition}
	\end{align}
	For \(X\in L(\A(\Gamma,\ell))\), define a directed subgraph
	\(F_X=([n],A_X)\) of \(\Gamma\) by
	\[
	(i,j)\in A_X
	\quad\Longleftrightarrow\quad
	X\subseteq H_{ij}^{\ell},
	\qquad
	H_{ij}^{\ell}:\ x_i-x_j=\ell .
	\]
	Thus \(A_X\) records precisely the deleted \(\ell\)-layer equations that hold
	identically on \(X\).
	By Lemma~\ref{lem:restriction-strata}, for any
	\(\bm y\in M_p((\A(\Gamma,\ell))^X)\cap
	M_p(\C_n^{m,\ell}(\Gamma))\), among the deleted \(\ell\)-layer equations we
	have the exact condition
	\begin{equation*}
		y_i-y_j=\ell
		\quad\Longleftrightarrow\quad
		(i,j)\in A_X
		\qquad ((i,j)\in A(\Gamma)).
	\end{equation*}

	We next translate this exact condition into graph-theoretic restrictions on
	\(A_X\).  Two arcs with a common head, or two arcs with a common tail, force
	two coordinates to be equal, contradicting the zero-layer hyperplanes.  Thus
	every vertex has indegree and outdegree at most one in \(A_X\).  Directed
	cycles are also impossible: summing around a cycle of length \(s\) gives
	\(s\ell=0\) in \(\mathbb F_p\), while \(0<s\ell\le nm<p\).  Hence every
	contributing \(A_X\) is a vertex-disjoint union of directed paths, with
	isolated vertices allowed.

	It remains to decide which path lengths are compatible with the remaining
	Catalan hyperplanes.  If \(1\le \ell\le \lfloor m/2\rfloor\), a directed path
	\(i\to j\to h\) would force \(y_i-y_h=2\ell\).  Since \(2\ell\le m\) and
	\(2\ell\ne \ell\), this lies on a remaining hyperplane of
	\(\C_n^{m,\ell}(\Gamma)\).
	Thus only directed matchings can contribute in this range.  If
	\(\lfloor m/2\rfloor+1\le \ell\le m\), then vertices separated by \(s\ge2\)
	selected arcs in one directed path have coordinate difference \(s\ell>m\), and
	\(s\ell\le(n-1)m<p-m\).  Therefore neither orientation gives a forbidden
	difference in \(\{1,\ldots,m\}\), so directed path covers are allowed.

	Conversely, let \(S\subseteq A(\Gamma)\)
	be a directed matching in the first range, or a directed path cover in the
	second range, and put \(X=\bigcap_{(i,j)\in S}H_{ij}^{\ell}\).  Then
	\(A_X=S\): in
	the matching case this follows from vertex-disjointness, while in the
	path-cover case nonconsecutive vertices in one path have difference \(s\ell\)
	or \(p-s\ell\) with \(s\ge2\), and different path components have independent
	additive constants.  Combining the necessity above with this converse, the
	summands in \eqref{eq:deleted-layer-decomposition} to be counted are exactly
	those indexed by directed matchings of \(\Gamma\) in the first range, and by
	directed path covers of \(\Gamma\) in the second range.

	Having identified the contributing \(X\)'s, we now show that, in either
	range, the contribution of such an \(X\) to
	\eqref{eq:deleted-layer-decomposition} depends only on \(|A_X|\),
	equivalently on \(k=n-|A_X|\).  Fix such an \(X\) and suppose
	\(|A_X|=n-k\).  The \(k\) directed path components of \(([n],A_X)\)
	can be written as
	\[
	P_r:\ v_{r,0}\to v_{r,1}\to\cdots\to v_{r,a_r},
	\qquad 1\le r\le k,
	\]
	where \(a_r=0\) is allowed and \(\sum_r a_r=n-k\).  In the placement model,
	\(P_r\) contributes a rigid block: once the box occupied by \(v_{r,0}\) is
	chosen, \(v_{r,a}\) must occupy the box \(a\ell\) steps counterclockwise from
	it.  Thus the block has length \(1+\ell a_r\).  The intermediate boxes are
	internal to this block: placing a vertex from another component there would
	create a positive coordinate difference in \([1,\ell-1]\), hence a remaining
	Catalan hyperplane.  Distinct blocks must be
	separated by at least \(m\) empty boxes.  Indeed, if two consecutive blocks
	are separated by fewer than \(m\) empty boxes, then the two endpoint vertices
	facing this gap have an oriented difference \(d\in[1,m]\).  If \(d\ne \ell\),
	the corresponding hyperplane is still present in the arrangement; if
	\(d=\ell\), it is either still present or is a deleted \(\ell\)-layer
	hyperplane not
	selected in \(A_X\).  In both cases this contradicts the exact stratum
	condition above.  Hence the total mandatory length is
	\[
	N=\sum_{r=1}^{k}(1+\ell a_r)+mk
	=\ell(n-k)+(m+1)k.
	\]
	Because \(\ell\le m\), we have \(N\le mn+k\le n(m+1)<p\).
	Since \(N-k+1=\theta_n^{m,\ell}(k)\), Lemma~\ref{lem:cyclic-gap-count} gives the
	contribution of this fixed \(X\):
	\[
	\#\bigl(M_p((\A(\Gamma,\ell))^X)\cap M_p(\C_n^{m,\ell}(\Gamma))\bigr)
	=
	p\bigl(p-\theta_n^{m,\ell}(k)\bigr)^{\underline{k-1}}.
	\]
	Summing over the directed matchings in the first range and over the directed
	path covers in the second range gives
	\[
	\#M_p(\C_n^{m,\ell}(\Gamma))
	=
	\begin{cases}
		\displaystyle
		p\sum_{k=\lceil n/2\rceil}^{n}
		\operatorname{mat}_{n-k}(\Gamma)
		\bigl(p-\theta_n^{m,\ell}(k)\bigr)^{\underline{k-1}},
		& \text{if } 1\le \ell\le \lfloor m/2\rfloor,\\[1.2ex]
		\displaystyle
		p\sum_{k=1}^{n}
		\pathnum_k(\Gamma)
		\bigl(p-\theta_n^{m,\ell}(k)\bigr)^{\underline{k-1}},
		& \text{if } \lfloor m/2\rfloor+1\le \ell\le m.
	\end{cases}
	\]
	By the finite field method, these identities hold for
	\(\chi(\C_n^{m,\ell}(\Gamma),p)\) for all sufficiently large primes \(p\). Hence
	they hold as polynomial identities after replacing \(p\) by \(t\).
\end{proof}

For deletion of the \(m\)-layer, Theorem~\ref{main-2} turns integer linear
factorization into an elementary question about a path-cover polynomial.

For a simple digraph \(\Gamma\) on \([n]\), put
\[
P(\Gamma,z)=\sum_{k=1}^{n}\pathnum_k(\Gamma)z^{\underline{k-1}}.
\]

\begin{coro}\label{cor:factorization-criterion}
	Let \(m\ge1\), and let \(\Gamma\) be a simple digraph on \([n]\).
	Then there exist integers \(\alpha_1,\dots,\alpha_{n-1}\) such that
	\[
	P(\Gamma,z)=\prod_{j=1}^{n-1}(z+\alpha_j)
	\]
	if and only if the same integers satisfy
	\[
	\chi(\mathcal C_n^{m,m}(\Gamma),t)
	=
	t\prod_{j=1}^{n-1}(t-mn-1+\alpha_j).
	\]
	Consequently, for fixed \(\Gamma\), whether
	\(\chi(\mathcal C_n^{m,m}(\Gamma),t)\) factors into integer linear factors is
	independent of \(m\). Changing \(m\) only translates the nonzero roots.
\end{coro}

\begin{proof}
	By Theorem~\ref{main-2}(2) with \(\ell=m\),
	\[
	\chi(\C_n^{m,m}(\Gamma),t)=tP(\Gamma,t-mn-1).
	\]
	Therefore, for integers \(\alpha_1,\dots,\alpha_{n-1}\),
	\[
	P(\Gamma,z)=\prod_{j=1}^{n-1}(z+\alpha_j)
	\]
	is equivalent, after the substitution \(z=t-mn-1\), to
	\[
	\chi(\C_n^{m,m}(\Gamma),t)
	=
	t\prod_{j=1}^{n-1}(t-mn-1+\alpha_j).
	\]
	This proves the stated criterion.  In particular, the roots of
	\(\chi(\C_n^{m,m}(\Gamma),t)/t\) are obtained from the roots of
	\(P(\Gamma,z)\) by adding \(mn+1\).  Hence, for fixed \(\Gamma\), the
	existence of a factorization into integer linear factors is independent of \(m\).
\end{proof}

\subsection{Multilayer deletions}
\label{sec-simultaneous-translated-deletions}

The preceding proof extends to simultaneous deletions in translated layers.  The
coefficients are no longer ordinary matching or path-cover numbers, but
weighted counts of admissible sets of deleted gain-labeled arcs.

Fix \(m\ge1\), let \(S\subseteq[m]\), and let
\(\Gamma_s=([n],A_s)\) be a simple digraph for each \(s\in S\).  Put
\(\Gamma_\bullet=(\Gamma_s)_{s\in S}\), and write
\[
\C_n^{m,S}(\Gamma_\bullet)
=
\C_n^m
\setminus
\{x_i-x_j=s:\ s\in S,\ (i,j)\in A_s\}.
\]
A \emph{deleted gain-labeled arc} is a triple \((i,j;s)\) with
\(s\in S\) and \((i,j)\in A_s\).  A set \(F\) of deleted gain-labeled arcs is
viewed as a spanning gain graph on \([n]\), with parallel triples of distinct
gains allowed.  Weak connectivity ignores orientations and gains; \(V(C)\)
denotes the vertex set of a weak component \(C\).

\begin{defi}\label{def:admissible-gain-labeled-arcs}
	A set \(F\) of deleted gain-labeled arcs is \emph{admissible} if the following
	conditions hold.
	\begin{itemize}
		\item[(1)] On every weak component \(C\), there is an integer potential
		\(\phi_C:V(C)\to\mathbb Z\) satisfying
		\[
		\phi_C(i)-\phi_C(j)=s
		\qquad\text{for every }(i,j;s)\in F.
		\]
		\item[(2)] The values of \(\phi_C\) are distinct on \(V(C)\).
		\item[(3)] If \(u,v\in V(C)\), \(u\ne v\), and
		\(\phi_C(u)-\phi_C(v)\in[1,m]\), then
		\[
		(u,v;\phi_C(u)-\phi_C(v))\in F.
		\]
	\end{itemize}
\end{defi}

Condition (3) is the closure condition for restriction strata: every forced
difference \(x_u-x_v=d\) with \(d\in[1,m]\) must correspond to a deleted
hyperplane recorded in the stratum; otherwise the stratum is empty.

For an admissible set \(F\), the potential on each weak component \(C\) is
unique up to adding a constant, so the quantity
\[
w(C)=\max_{v\in V(C)}\phi_C(v)-\min_{v\in V(C)}\phi_C(v)
\]
is well defined; define the \emph{span} of \(C\) to be \(w(C)\).  Let \(c(F)\)
be the number of weak components of \(F\), including isolated vertices, and set
\[
W(F)=\sum_C w(C),
\]
where the sum runs over all weak components.

	When \(S=\{\ell\}\), the admissibility conditions
	specialize to the restrictions already appearing in the single-layer
	argument: the admissible sets are directed matchings for
	\(1\le \ell\le\lfloor m/2\rfloor\), and directed path covers for
	\(\lfloor m/2\rfloor<\ell\le m\).  If such a set has \(k\) path components,
	including isolated vertices, then \(c(F)=k\) and \(W(F)=\ell(n-k)\).  Hence the
	shift in Theorem~\ref{thm:simultaneous-translated-deletions} becomes
	\[
	W(F)+m c(F)+1=\ell(n-k)+mk+1=\theta_n^{m,\ell}(k),
	\]
	which gives exactly the two formulas in Theorem~\ref{main-2}.

Before proving Theorem~\ref{thm:simultaneous-translated-deletions}, we
illustrate the closure condition in the admissibility conditions above.

\begin{example}\label{ex:multilayer-closure}
	Let \(m=3\), and consider the deleted gain-labeled arcs \((2,1;2)\) and
	\((3,2;1)\).  These two equations force potentials
	\[
	\phi(1)=0,\qquad \phi(2)=2,\qquad \phi(3)=3.
	\]
	The induced difference \(\phi(3)-\phi(1)=3\) lies in the interval
	\([1,3]\).  Therefore these two arcs alone do not form an admissible set:
	either the hyperplane \(x_3-x_1=3\) remains, making the stratum empty, or it
	is deleted but missing from the stratum that records exactly which deleted
	hyperplanes vanish.  If the deleted system also contains \((3,1;3)\), then
	the three arcs form an admissible component of span \(3\).
\end{example}

In the placement model, a component \(C\) is placed by putting its vertices at
their potential values.  For the component in
Example~\ref{ex:multilayer-closure}, the potentials are \(0,2,3\), so the
vertices occupy the positions \(0,2,3\).  The interval from the minimum to the
maximum potential forms one block of length \(w(C)+1\); empty positions inside
this interval are not part of the freely distributed gaps.  In the figure, the
green arcs are the two initially chosen deleted arcs, while the red arc is the
closure arc forced by the induced difference \(\phi(3)-\phi(1)=3\).
Figure~\ref{fig:gain-labeled-component-placement} illustrates this placement.

\begin{figure}[H]
	\centering
	\begin{tikzpicture}[x=1.32cm,y=1cm]
		\draw[figgray, thick] (-0.35,0) -- (6.9,0);
		\node[fig vertex] (u) at (0,0) {\(1\)};
		\node[fig vertex] (v) at (2,0) {\(2\)};
		\node[fig vertex] (w) at (3,0) {\(3\)};
		\node[font=\scriptsize] at (0,-0.78) {\(\phi=0\)};
		\node[font=\scriptsize] at (2,-0.78) {\(\phi=2\)};
		\node[font=\scriptsize] at (3,-0.78) {\(\phi=3\)};
		\draw[fig arrow, figgreen]
		(v) to[out=-145,in=-35,looseness=1.05]
		node[midway, below=2pt, font=\scriptsize, fill=white, inner sep=1pt] {\(2\)} (u);
		\draw[fig arrow, figgreen]
		(w) to[out=-145,in=-35,looseness=1.05]
		node[midway, below=2pt, font=\scriptsize, fill=white, inner sep=1pt] {\(1\)} (v);
		\draw[fig arrow, figred]
		(w) to[out=135,in=45,looseness=1.05]
		node[midway, above=2pt, font=\scriptsize, fill=white, inner sep=1pt] {\(3\)} (u);
		\draw[decorate, decoration={brace, mirror, amplitude=4pt}, figgold]
		(0,-1.2) -- node[below=4pt, font=\scriptsize] {span \(w(C)=3\)} (3,-1.2);
		\node[fig gap, dashed, fill=white, minimum width=1.45cm] (sep) at (4.65,0)
		{\(3\) empty};
		\node[fig block, minimum width=0.95cm] at (6.25,0) {next};
	\end{tikzpicture}
	\caption{The placement model for the admissible component in Example~\ref{ex:multilayer-closure}.}
	\label{fig:gain-labeled-component-placement}
\end{figure}
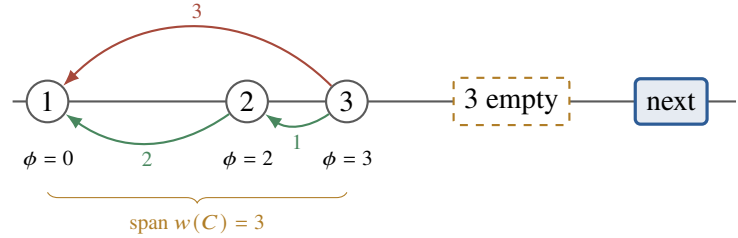

\begin{proof}[Proof of Theorem~\ref{thm:simultaneous-translated-deletions}]
	Let \(p\) be a sufficiently large prime with \(p>n(m+1)\), and let
	\(\mathcal A_D\) be the arrangement over \(\mathbb F_p\) formed by the deleted
	gain hyperplanes
	\[
	x_i-x_j=s
	\qquad (s\in S,\ (i,j)\in A_s).
	\]
	By the finite field method, it is enough to count
	\(M_p(\C_n^{m,S}(\Gamma_\bullet))\).  Decomposing \(\mathbb F_p^n\) into
	restriction strata for \(\mathcal A_D\), we get
	\begin{align}
	M_p(\C_n^{m,S}(\Gamma_\bullet))
	=
	\bigsqcup_{X\in L(\mathcal A_D)}
	\bigl(M_p((\mathcal A_D)^X)\cap
	M_p(\C_n^{m,S}(\Gamma_\bullet))\bigr).
	\label{eq:multilayer-deleted-strata}
	\end{align}
	For \(X\in L(\mathcal A_D)\), define
	\[
	F_X
	=
	\{(i,j;s):s\in S,\ (i,j)\in A_s,\
	X\subseteq \{x_i-x_j=s\}\}.
	\]
	If \(\bm y\in M_p((\mathcal A_D)^X)\cap
	M_p(\C_n^{m,S}(\Gamma_\bullet))\), then Lemma~\ref{lem:restriction-strata}
	gives the exact condition
	\[
	y_i-y_j=s
	\quad\Longleftrightarrow\quad
		(i,j;s)\in F_X
		\qquad (s\in S,\ (i,j)\in A_s).
	\]
	
	Suppose the stratum indexed by \(X\) is nonempty.  On each weak component of
	\(F_X\), define potentials by signed gain sums from a base vertex.  These
	potentials are well defined: the gain sum around any simple cycle is
	congruent to \(0\) modulo \(p\), has absolute value at most \(mn<p\), and
	therefore is \(0\) in \(\mathbb Z\).  Equal potentials would force equal
	coordinates, contradicting the zero-layer hyperplanes.  If vertices \(u,v\)
	in one component have \(\phi(u)-\phi(v)=d\in[1,m]\), then the stratum forces
	\(x_u-x_v=d\); this hyperplane must be deleted and, by the exact condition,
	must be recorded in \(F_X\).  Hence \(F_X\) is admissible.
	
	Conversely, let \(F\) be admissible.  Condition (1) gives a potential on each
	weak component.  Choosing one additive constant \(a_C\) for each component and
	setting \(x_v=a_C+\phi_C(v)\) shows that
	\[
	X=\bigcap_{(i,j;s)\in F}\{x_i-x_j=s\}
	\]
	is nonempty, hence \(X\in L(\mathcal A_D)\).  By construction,
	\(F\subseteq F_X\).  For the reverse inclusion, suppose that a deleted
	gain hyperplane \(x_u-x_v=s\) contains \(X\).  If \(u\) and \(v\) lie in
	different weak components, the component of \(u\) can be translated
	independently of the others, a contradiction.  Hence \(u\) and \(v\) lie in
	the same weak component \(C\), and \(X\) forces
	\(x_u-x_v=\phi_C(u)-\phi_C(v)\).  Since
	\(|\phi_C(u)-\phi_C(v)-s|\le mn<p\), the congruence
	\(\phi_C(u)-\phi_C(v)\equiv s\pmod p\) implies the integer equality
	\(\phi_C(u)-\phi_C(v)=s\).  Condition (3) gives \((u,v;s)\in F\).  Thus
	\(F_X=F\), and the summands in \eqref{eq:multilayer-deleted-strata} to be
	counted are exactly those indexed by admissible sets \(F\).

	Fix an admissible set \(F\).  We regard each weak component \(C\) as a rigid
	block.  Normalize its potential so that the minimum is \(0\).  Once the
	starting position of \(C\) is chosen, every vertex \(v\in V(C)\) must be placed
	at offset \(\phi_C(v)\); hence \(C\) occupies an interval of length
	\(w(C)+1\).  The empty boxes inside this interval are part of the block, not
	free gaps.  Indeed, consecutive occupied potential values in \(C\) differ by
	at most \(m\), since otherwise no gain edge could cross the gap,
	contradicting connectedness.  If a vertex from another component were inserted
	into such an internal gap, then it would lie within distance \(m\) of one of
	the two vertices of \(C\) adjacent to that gap.  Thus this outside vertex and
	that adjacent vertex of \(C\) would have an oriented difference in \([1,m]\).
	This gives either a remaining Catalan hyperplane or an unselected deleted gain
	hyperplane, both excluded by the exact condition.

	Attach \(m\) empty boxes after each component block.  This prevents
	Catalan-range differences between distinct components; inside one component,
	admissibility records exactly the forced differences in \([1,m]\).

	The total mandatory length, including these separating boxes, is
	\[
	N=\sum_C (w(C)+1+m)=W(F)+(m+1)c(F).
	\]
	A spanning tree of a component \(C\) shows that
	\(w(C)\le m(|V(C)|-1)\).  Hence \(N\le mn+c(F)\le n(m+1)<p\), so the cyclic
	gap lemma applies.
	Thus we have \(c(F)\) distinguishable blocks of total mandatory length \(N\),
	placed on the cycle of \(p\) boxes labeled by \(0,1,\ldots,p-1\), with
	arbitrary cyclic order.
	Lemma~\ref{lem:cyclic-gap-count} therefore gives
	\[
	p\bigl(p-N+c(F)-1\bigr)^{\underline{c(F)-1}}
	=
	p\bigl(p-W(F)-m\,c(F)-1\bigr)^{\underline{c(F)-1}}
	\]
	placements, exactly the points in the stratum indexed by \(F\).  Summing over
	all admissible sets \(F\), the finite field method gives
	\[
	\chi(\C_n^{m,S}(\Gamma_\bullet),p)
	=
	p\sum_F
	\bigl(p-W(F)-m\,c(F)-1\bigr)^{\underline{c(F)-1}}
	\]
	for all sufficiently large primes \(p\).  Hence the same identity holds as a
	polynomial identity after replacing \(p\) by \(t\).
\end{proof}

\begin{remark*}[Odd translated layers]
	The admissible-set formula has a simple form for odd translated layers on one
	side of the threshold \(\lfloor m/2\rfloor\).  Keep the notation
	\(\Gamma_\bullet=(\Gamma_s)_{s\in S}\); throughout this remark, \(c(F)\) and
	\(W(F)\) are as in Theorem~\ref{thm:simultaneous-translated-deletions}.
	First take
	\[
	S_-=\{s\in[m]:s\le\lfloor m/2\rfloor,\ s\equiv1\pmod 2\}.
	\]
	Thus \(\Gamma_\bullet=(\Gamma_s)_{s\in S_-}\).  Let
	\(\mathcal M(\Gamma_\bullet)\) be the set of pairwise vertex-disjoint
	gain-labeled arcs \((i,j;s)\), with \(s\in S_-\) and \((i,j)\in A_s\).
	Thus
	\[
	\chi(\C_n^{m,S_-}(\Gamma_\bullet),t)
	=
	t\sum_{F\in\mathcal M(\Gamma_\bullet)}
	\bigl(t-W(F)-m\,c(F)-1\bigr)^{\underline{c(F)-1}}.
	\]
	
	Next take
	\[
	S_+=\{s\in[m]:\lfloor m/2\rfloor<s\le m,\ s\equiv1\pmod 2\}.
	\]
	Thus \(\Gamma_\bullet=(\Gamma_s)_{s\in S_+}\).  Let
	\(\mathcal P(\Gamma_\bullet)\) be the set of gain-labeled arc sets
	\((i,j;s)\), with \(s\in S_+\) and \((i,j)\in A_s\), whose spanning gain
	graph is a disjoint union of directed paths, with isolated vertices allowed.
	Thus
	\[
	\chi(\C_n^{m,S_+}(\Gamma_\bullet),t)
	=
	t\sum_{F\in\mathcal P(\Gamma_\bullet)}
	\bigl(t-W(F)-m\,c(F)-1\bigr)^{\underline{c(F)-1}}.
	\]
	To see this, suppose first that \(F\) is admissible.  Since all gains under
	consideration are odd, two arcs with the same head or the same tail force
	either equal potentials, contradicting condition (2), or a nonzero even
	difference in \([1,m]\), contradicting the closure condition (3).  Directed
	cycles are also impossible, since the sum of the positive gains around such a
	cycle cannot be zero.  Hence every nontrivial component is a directed path.
	For \(S=S_-\), a directed path of length two forces an even endpoint
	difference in \([1,m]\), again contradicting condition (3); thus every
	nontrivial component is a single arc.  Conversely, every set in
	\(\mathcal M(\Gamma_\bullet)\) is admissible.  For \(S=S_+\), every
	nonconsecutive difference along a directed path is a sum of at least two gains
	from \(S_+\), hence is greater than \(m\).  Thus the admissible sets are
	exactly the path sets counted by \(\mathcal P(\Gamma_\bullet)\).
\end{remark*}

\section{Path-cover consequences}\label{sec-path-cover-consequences}

We now specialize the single-layer formula to the path-cover regime.  This gives
a Lah-inversion relation for complements and several computable consequences,
including extremal cases, bipartite orientations, and formulas obtained from
known factorizations.

\subsection{Lah inversion}\label{sec-path-cover-transform}

We compute the case \(\ell=m\) by inclusion--exclusion over the \(m\)-layer
hyperplanes that remain after the deletion, and then compare the result with
Theorem~\ref{main-2}.  These remaining hyperplanes are indexed by the complement
of \(\Gamma\) in the complete loopless digraph.

For a simple digraph \(\Gamma=([n],A(\Gamma))\), write
\(\bar\Gamma=([n],A(\bar\Gamma))\) for its complement in the complete loopless
digraph on \([n]\), so that
\[
A(\bar\Gamma)
=
\{(u,v)\in[n]\times[n]\mid u\ne v,\ (u,v)\notin A(\Gamma)\}.
\]

\begin{theorem}\label{thm:digraph-catalan-l-equals-m}
	Let \(m\ge 1\), and let \(\Gamma=([n],A(\Gamma))\) be a simple digraph.  Let
	\(\bar\Gamma=([n],A(\bar\Gamma))\) be the complement of \(\Gamma\) in the complete
	loopless digraph.  Then
	\[
	\chi(\C_n^{m,m}(\Gamma),t)
	=
	t\sum_{k=1}^{n}
	(-1)^{n-k}
	\pathnum_k(\bar\Gamma)
	(t-mn+k-1)^{\underline{k-1}}.
	\]
\end{theorem}

\begin{proof}
	Let \(p\) be a sufficiently large prime with \(p>n(m+1)\).  By the finite field method, it is
	enough to compute \(\#M_p(\C_n^{m,m}(\Gamma))\).  Since
	\(\C_n^{m,m}(\Gamma)\) is obtained from the full \(m\)-Catalan arrangement by
	deleting the \(m\)-layer hyperplanes indexed by \(A(\Gamma)\), we can write
	\[
	\C_n^{m,m}(\Gamma)
	=
	\mathcal C_n^{m-1}
	\cup
	\{x_i-x_j=m:(i,j)\in A(\bar\Gamma)\}.
	\]
	Thus it remains to impose, on \(M_p(\mathcal C_n^{m-1})\), the avoidance of
	the \(m\)-layer hyperplanes indexed by \(A(\bar\Gamma)\).  For
	\(S\subseteq A(\bar\Gamma)\), set
	\[
	f(S)
	=
	\{(x_1,\ldots,x_n)\in M_p(\mathcal C_n^{m-1}):x_i-x_j=m
	\text{ for every }(i,j)\in S\}.
	\]
	Here \(S\) records only the
	\(m\)-layer equations we impose, not necessarily all such equations satisfied
	by the point; the latter are accounted for by inclusion--exclusion.
	Inclusion--exclusion gives
	\[
	\#M_p(\C_n^{m,m}(\Gamma))
	=
	\sum_{S\subseteq A(\bar\Gamma)}
	(-1)^{|S|}\#f(S).
	\]
	We first determine when \(f(S)\) can be nonempty.  If \(S\) contains two arcs
	with the same tail, say \(i\to j\) and \(i\to h\), then the equations
	\(x_i-x_j=m\) and \(x_i-x_h=m\) force \(x_j=x_h\), impossible in
	\(M_p(\mathcal C_n^{m-1})\).  The same argument excludes two arcs with the
	same head.  Hence every vertex has indegree and outdegree at most one in
	\(([n],S)\), so \(([n],S)\) is a disjoint union of directed paths and directed
	cycles, with isolated vertices allowed.  Directed cycles are impossible:
	summing \(x_i-x_j=m\) around a cycle of length \(s\) gives \(sm=0\) in
	\(\mathbb F_p\), while \(0<sm\le mn<p\).  Thus only
	directed path covers of \(\bar\Gamma\) can contribute.
	
	Fix a directed path cover \(S\) with \(k\) directed paths, so \(|S|=n-k\).  In
	the placement model, write the paths as
	\[
	P_r:\ v_{r,0}\to v_{r,1}\to\cdots\to v_{r,a_r},
	\qquad 1\le r\le k,
	\]
	where \(a_r=0\) is allowed and \(\sum_r a_r=n-k\).  Each \(P_r\) forms a
	rigid block: after the box occupied by \(v_{r,0}\) is chosen, the vertex
	\(v_{r,a}\) is placed \(am\) steps counterclockwise from \(v_{r,0}\).  Thus
	the block has length \(1+a_r m\), with \(m-1\) empty boxes between consecutive
	vertices of the path.  No other vertex can occupy an intermediate box,
	because it would then differ from a neighboring path vertex by an element of
	\([1,m-1]\), contrary to membership in \(M_p(\mathcal C_n^{m-1})\).  Since
	\(p>n(m+1)\), these prescribed positions are
	distinct and do not wrap around the \(p\)-cycle.  Attach \(m-1\) empty boxes
	after each block to separate distinct blocks.  The total mandatory length is
	\[
	N=\sum_{r=1}^k(1+a_r m)+k(m-1)=mn.
	\]
	Since \(N=mn<p\), Lemma~\ref{lem:cyclic-gap-count} gives
	\[
	\#f(S)=p(p-mn+k-1)^{\underline{k-1}}
	\]
	for this fixed path cover.
	Summing over all path covers of \(\bar\Gamma\) with \(k\) directed paths and
	using \(|S|=n-k\), we obtain
	\[
	\#M_p(\C_n^{m,m}(\Gamma))
	=
	p\sum_{k=1}^{n}
	(-1)^{n-k}
	\pathnum_k(\bar\Gamma)
	(p-mn+k-1)^{\underline{k-1}}.
	\]
	By the finite field method, the same formula holds for
	\(\chi(\C_n^{m,m}(\Gamma),p)\).  Since this identity holds for all
	sufficiently large primes \(p\), it holds as a polynomial identity after
	replacing \(p\) by \(t\).
\end{proof}

For \(\ell=m\), the path-cover expansion in Theorem~\ref{main-2}
and the inclusion--exclusion expansion in
Theorem~\ref{thm:digraph-catalan-l-equals-m} are two expansions of the
same characteristic polynomial.  Hence, for every simple digraph \(\Gamma\),
\begin{equation}\label{eq:l-equals-m-path-cover-comparison}
\sum_{h=1}^{n}
\pathnum_h(\Gamma)(t-mn-1)^{\underline{h-1}}
=
\sum_{k=1}^{n}
(-1)^{n-k}
\pathnum_k(\bar\Gamma)(t-mn+k-1)^{\underline{k-1}}.
\end{equation}
Let
\[
L(k,h)=\binom{k-1}{h-1}\frac{k!}{h!}
\qquad (1\le h\le k)
\]
be the unsigned Lah numbers.

\begin{coro}\label{cor:path-cover-complement-transform}
	For every simple digraph \(\Gamma\) on \([n]\),
	\begin{equation}\label{eq:path-cover-complement-transform}
		\pathnum_h(\bar\Gamma)
		=
		\sum_{k=h}^{n}
		(-1)^{n-k}
		L(k,h)
		\pathnum_k(\Gamma)
		\qquad (1\le h\le n).
	\end{equation}
\end{coro}

\begin{proof}
	Set \(y=t-mn\).  By \eqref{eq:l-equals-m-path-cover-comparison}, with
	\(\Gamma\) replaced by \(\bar\Gamma\),
	\[
	\sum_{h=1}^{n}
	\pathnum_h(\bar\Gamma)(y-1)^{\underline{h-1}}
	=
	\sum_{k=1}^{n}
	(-1)^{n-k}
	\pathnum_k(\Gamma)(y+k-1)^{\underline{k-1}}.
	\]
	Using
	\[
	(y+k-1)^{\underline{k-1}}
	=
	\sum_{h=1}^{k}
	L(k,h)
	(y-1)^{\underline{h-1}},
	\]
	and comparing the coefficients of
	\((y-1)^{\underline{h-1}}\), we obtain the desired identity.
\end{proof}

Equivalently, one may count \(h\)-path covers in the complete loopless digraph
that avoid all arcs of \(\Gamma\); a fixed \(k\)-component path cover using
arcs of \(\Gamma\) can be completed to an \(h\)-component path cover in
\(L(k,h)\) ways.

\begin{remark*}
	For simple graphs, this operator form is the path-cover duality relation
	studied in~\cite{jumadildayev2025duality}.  The same transformation has the
	following digraph form.  Let \(D=d/dz\).  For a digraph \(\Gamma\) on \([n]\), define
	\[
	\Path(\Gamma,z)=\sum_{k=1}^{n}\pathnum_k(\Gamma)z^k.
	\]
	Then \eqref{eq:path-cover-complement-transform} is equivalent to
	\[
	\Path(\bar\Gamma,z)
	=
	(-1)^n e^{zD^2}\Path(\Gamma,-z),
	\]
	where \(e^{zD^2}\) is understood by its finite power series on polynomials.
	Indeed, the monomial identity
	\[
	e^{zD^2}z^k
	=
	\sum_{h=1}^{k}
	L(k,h)z^h
	\]
	shows that the coefficient comparison in
	\eqref{eq:path-cover-complement-transform} is exactly this operator identity.
\end{remark*}

\subsection{Extremal cases}\label{sec-coefficient-formulas}

We give computable consequences of Theorem~\ref{main-2}.  Under the present
deletion convention, \(\Gamma\) indexes the deleted \(\ell\)-th layer hyperplanes,
so the coefficients are matching and path-cover numbers of \(\Gamma\).  The
examples below cover the case \(\ell=m\), a complete bipartite
orientation with compact region counts, and factored arrangements from which
path-cover numbers can be extracted.

Two extremal digraphs provide checks on the formula and connect it to
standard deformations.
Let \(\mathcal T_n\) be the transitive tournament with arcs \((i,j)\) for
\(i<j\).  A path cover of \(\mathcal T_n\) is the same as a set partition of
\([n]\), since each block has a unique increasing directed path.  Hence
\(\pathnum_k(\mathcal T_n)=\left\{\begin{smallmatrix}n\\ k\end{smallmatrix}\right\}\).
At \(\ell=m\), Theorem~\ref{main-2} gives
\[
\chi(\C_n^{m,m}(\mathcal T_n),t)
=
t\sum_{k=1}^{n}
\left\{\begin{matrix}n\\ k\end{matrix}\right\}
(t-mn-1)^{\underline{k-1}}
=
t(t-mn)^{n-1}.
\]
Under the coordinate change \(x\mapsto-x\), this arrangement becomes the
standard extended \(m\)-Shi arrangement.

At the opposite extreme, let \(\mathcal D_n\) be the complete loopless digraph
on \([n]\).  Its path-cover numbers are the unsigned Lah numbers
\[
L(n,k)=\binom{n-1}{k-1}\frac{n!}{k!}
\qquad (1\le k\le n),
\]
which, under Gonzales's nonadjacency convention, form the empty-graph case of
linear graph partitions~\cite{gonzales2022}.  Since
deleting all \(m\)-th layer hyperplanes from the full \(m\)-Catalan arrangement
leaves the full \((m-1)\)-Catalan arrangement, Theorem~\ref{main-2} gives the
identity
\[
t\sum_{k=1}^{n}
L(n,k)(t-mn-1)^{\underline{k-1}}
=
t(t-(m-1)n-1)^{\underline{n-1}}.
\]

\subsection{Bipartite orientations}

Let \(\mathcal K_{a,b}\) have vertex classes \(A\) and \(B\),
with positive sizes \(|A|=a\), \(|B|=b\), and all arcs directed from \(A\) to
\(B\).  Put \(n=a+b\).  Since \(\mathcal K_{a,b}\) has no directed path of
length \(2\), a path cover is just a directed matching together with isolated vertices, and
the number of such arc sets with \(q\) arcs is
\(\binom aq\binom bq q!\).  Therefore, for every \(1\le \ell\le m\),
\[
\chi(\C_n^{m,\ell}(\mathcal K_{a,b}),t)
=
t\sum_{q=0}^{\min\{a,b\}}
\binom aq\binom bq q!\,
\bigl(t-\theta_n^{m,\ell}(n-q)\bigr)^{\underline{n-q-1}}.
\]
When \(\ell=m\), this example also gives compact region counts.  Let
\[
D=\mathbb R(1,\dots,1).
\]
Every hyperplane \(x_i-x_j=c\) is invariant under translation by \(D\), so we
count regions in the essentialization, that is, in \(\mathbb R^n/D\),
equivalently after fixing one coordinate.  Write
\[
\chi(\C_n^{m,m}(\mathcal K_{a,b}),t)=t\psi(t).
\]
Then \(\psi(t)\) is the characteristic polynomial of the essentialization.
By Zaslavsky's formulas~\cite{zaslavsky1975},
\[
r_{\mathrm{ess}}=(-1)^{n-1}\psi(-1),
\qquad
b_{\mathrm{ess}}=(-1)^{n-1}\psi(1).
\]
Thus
\[
r_{\mathrm{ess}}
=
\sum_{q=0}^{\min\{a,b\}}
(-1)^q\binom aq\binom bq q!\,
((m+1)n-q)^{\underline{n-q-1}},
\]
and
\[
b_{\mathrm{ess}}
=
\sum_{q=0}^{\min\{a,b\}}
(-1)^q\binom aq\binom bq q!\,
((m+1)n-q-2)^{\underline{n-q-1}}.
\]
Set \(X=(m+1)n\).  We use the identity
\begin{equation}\label{eq:bipartite-vandermonde}
	\sum_{q=0}^{\min\{a,b\}}
	(-1)^q\binom aq\binom bq q!\,
	(X-q)^{\underline{a+b-q-1}}
	=
	(X-b)^{\underline a}(X-a)^{\underline{b-1}}.
\end{equation}
This is a direct Chu--Vandermonde convolution, after factoring
\((X-a)^{\underline{b-1}}\).  Applying
\eqref{eq:bipartite-vandermonde} with \(X\) and with \(X-2\) gives
\[
r_{\mathrm{ess}}
=
((m+1)n-b)^{\underline a}
((m+1)n-a)^{\underline{b-1}},
\]
and
\[
b_{\mathrm{ess}}
=
((m+1)n-b-2)^{\underline a}
((m+1)n-a-2)^{\underline{b-1}}.
\]

\subsection{Factorization formulas}

Theorem~\ref{main-2} can also be used in the reverse direction.  Taking
\(m=\ell=1\), its second part gives
\[
\chi(\C_n^{1,1}(\Gamma),t)
=
t\sum_{k=1}^{n}
\pathnum_k(\Gamma)(t-n-1)^{\underline{k-1}}.
\]
Thus, whenever the same characteristic polynomial is known to factor into
integer linear factors, the path-cover numbers of \(\Gamma\) can be extracted
from its expansion in the falling-factorial basis.
Equivalently, for the path-cover polynomial \(P(\Gamma,z)\), a factorization
gives explicit coefficients: if
\[
P(\Gamma,z)=\prod_{j=1}^{n-1}(z+\alpha_j),
\]
then Newton's finite-difference extraction in the basis
\(\{z^{\underline{k-1}}\}\) gives
\[
\pathnum_k(\Gamma)
=
\frac{1}{(k-1)!}
\sum_{q=0}^{k-1}
(-1)^{k-1-q}\binom{k-1}{q}
\prod_{j=1}^{n-1}(q+\alpha_j)
\qquad (1\le k\le n).
\]
The next two propositions apply Athanasiadis' factorization theorems to the
arc set \(A\) of retained \(1\)-layer hyperplanes.

Let \(A(\mathcal T_n)=\{(i,j):1\le i<j\le n\}\) and
\(A(\mathcal D_n)=\{(i,j):1\le i\ne j\le n\}\).  We first use the one-sided
factorization of Athanasiadis~\cite[Theorem~3.4]{athanasiadis1996}.  Let
\(A\subseteq A(\mathcal T_n)\) be such that if \((i,j)\in A\) and \(j<r\), then
\((i,r)\in A\).  Equivalently, for each
fixed tail \(i\), the set \(\{j:(i,j)\in A\}\) is a terminal
interval in \(\{i+1,\dots,n\}\).  Let
\[
A(\Gamma)=A(\mathcal D_n)\setminus A.
\]
Then \(\Gamma=([n],A(\Gamma))\) indexes the deleted \(1\)-layer hyperplanes, and
the retained \(1\)-layer hyperplanes in \(\C_n^{1,1}(\Gamma)\) are exactly \(x_i-x_j=1\), if \((i,j)\in A\). For \(2\le j\le n\), set
\[
a_j=\#\{i<j:(i,j)\in A\}.
\] 
Athanasiadis' factorization gives
\[
\chi(\C_n^{1,1}(\Gamma),t)
=
t\prod_{j=2}^{n}(t-n+j-a_j-1).
\]
Comparing this product formula with the case \(m=\ell=1\) of
Theorem~\ref{main-2} yields the following path-cover formula.

\begin{prop}\label{prop:factorization-extraction-initial-intervals}
	Let \(A\subseteq A(\mathcal T_n)\) be such that if \((i,j)\in A\) and
	\(j<r\), then \((i,r)\in A\), and let
	\(\Gamma=([n],A(\mathcal D_n)\setminus A)\).  Then, for \(1\le k\le n\),
	\[
	\pathnum_k(\Gamma)
	=
	\frac{1}{(k-1)!}
	\sum_{q=0}^{k-1}
	(-1)^{k-1-q}
	\binom{k-1}{q}
	\prod_{j=2}^{n}(q+j-a_j).
	\]
	In particular, \(\pathnum_k(\Gamma)\) depends only on \(a_2,\dots,a_n\).
\end{prop}

\begin{proof}
	After setting \(z=t-n-1\), the comparison becomes
\[
\sum_{k=1}^{n}\pathnum_k(\Gamma)z^{\underline{k-1}}
=
\prod_{j=2}^{n}(z+j-a_j),
\]
	and the displayed formula follows by extracting the coefficient of
	\(z^{\underline{k-1}}\) through Newton's finite-difference formula.
\end{proof}

We also use Athanasiadis' directed factorization theorem
\cite[Theorem~3.9]{athanasiadis1996}.  In our orientation convention, its
hypotheses require that, for distinct \(i,j,k\) with \(i,j<k\),
\begin{enumerate}
	\item[(A1)] if \((j,i)\in A\), then \((k,i)\in A\) or \((j,k)\in A\); and
	\item[(A2)] if \((k,i),(j,k)\in A\), then \((j,i)\in A\).
\end{enumerate}
Let \(A\subseteq A(\mathcal D_n)\) satisfy these conditions.  For
\(2\le j\le n\), set
\[
a_j
=
\#\{i<j:(i,j)\in A\}
+
\#\{i<j:(j,i)\in A\}.
\]
Take \(\Gamma=([n],A(\mathcal D_n)\setminus A)\).  Athanasiadis' directed
factorization gives the same product form for
\(\chi(\C_n^{1,1}(\Gamma),t)\), and comparison with
Theorem~\ref{main-2} gives the following formula.

\begin{prop}\label{prop:directed-athanasiadis-extraction}
	Let \(A\subseteq A(\mathcal D_n)\) satisfy conditions (A1) and (A2), and
	let \(\Gamma=([n],A(\mathcal D_n)\setminus A)\).  Then, for \(1\le k\le n\),
	\[
	\pathnum_k(\Gamma)
	=
	\frac{1}{(k-1)!}
	\sum_{q=0}^{k-1}
	(-1)^{k-1-q}
	\binom{k-1}{q}
	\prod_{j=2}^{n}(q+j-a_j).
	\]
	In particular, \(\pathnum_k(\Gamma)\) depends only on \(a_2,\dots,a_n\).
\end{prop}

\begin{proof}
	The directed factorization gives the same product expression in the variable
	\(z=t-n-1\) as in the preceding proposition.  Applying the same
	finite-difference extraction to the falling-factorial expansion from
	Theorem~\ref{main-2} gives the stated formula.
\end{proof}

\section{Ish-type applications}\label{sec-ish-type-applications}

The final section applies the same finite-field stratification to Ish-type
arrangements.  The first family extends \(G\)-Ish arrangements by allowing
directed indexing data, while the second gives a fixed-index variant with a
degree-based characteristic polynomial.

\subsection{Directed \texorpdfstring{\(G\)}{G}-Ish arrangements}\label{sec-directed-ish}

Armstrong and Rhoades introduced the \(G\)-Ish arrangement by adding the
hyperplanes \(x_1-x_j=i\) indexed by edges \(ij\in G\) with \(i<j\)
\cite{armstrong2012}.  The following digraph version recovers their arrangement
when \(G\) is oriented increasingly.  Its characteristic polynomial is expressed
through path-cycle covers of a complementary indexing digraph.  We write the
arrangement as \(\Ish(\Gamma)\).

	Let \(\Gamma=([n],A(\Gamma))\) be a digraph with loops allowed and with
	\(A(\Gamma)\subseteq [n]\times([n]\setminus\{1\})\).
	Define
	\[
	\Ish(\Gamma)
	=
	\mathcal B_n
	\cup
	\{x_1-x_j=i:(i,j)\in A(\Gamma)\}.
	\]

For the rest of this section, complements are taken in the ambient indexing set
of hyperplanes \(x_1-x_j=i\):
\[
A(\bar\Gamma(1))
=
\{(i,j):1\le i\le n,\ j\neq 1\}\setminus A(\Gamma),
\qquad
\bar\Gamma(1)
=
([n],A(\bar\Gamma(1))).
\]
This is an indexing digraph; the first coordinate is the translation parameter
and the second is the constrained coordinate.  Loops may occur, since
\((i,i)\) indexes \(x_1-x_i=i\).

Let \(\Gamma=([n],A(\Gamma))\) be a digraph with loops allowed.  We call a
spanning subgraph \(F=([n],S)\), with \(S\subseteq A(\Gamma)\), a
\emph{directed path-cycle cover} of \(\Gamma\) if each weakly connected
component of \(F\) is either a directed path or a directed cycle.  An isolated
vertex is regarded as a directed path of length \(0\), and a loop is regarded
as a directed cycle of length \(1\).  Let \(\pc_r(\Gamma)\) denote the number
of directed path-cycle covers with exactly \(r\) directed path components.

\begin{theorem}\label{thm:directed-ish-indexed}
	Let \(\Gamma=([n],A(\Gamma))\) be a digraph with loops allowed and with
	\(A(\Gamma)\subseteq [n]\times([n]\setminus\{1\})\).
	Then
	\[
	\chi(\Ish(\Gamma),t)
	=
	t\sum_{k=1}^{n}
	\pc_k(\bar\Gamma(1))
	(t-n-1)^{\underline{k-1}}.
	\]
\end{theorem}

\begin{proof}
	Let \(p\) be a sufficiently large prime with \(p>2n\).  By the finite field method, it is
	enough to compute \(\#M_p(\Ish(\Gamma))\).  Let
	\[
	\A_{\bar\Gamma(1)}
	=
	\{x_1-x_j=i:(i,j)\in A(\bar\Gamma(1))\}.
	\]
	By the decomposition \eqref{decomposition}, after reduction modulo \(p\) we
	have
	\begin{equation}\label{condition}
	M_p(\Ish(\Gamma))
	=
	\bigsqcup_{X\in L(\A_{\bar\Gamma(1)})}
	\left(
	M_p(\Ish(\Gamma))\cap M_p((\A_{\bar\Gamma(1)})^X)
	\right).
	\end{equation}
	For \(X\in L(\A_{\bar\Gamma(1)})\), set
	\[
	A_X
	=
	\{(i,j)\in A(\bar\Gamma(1)):
	X\subseteq \{x_1-x_j=i\}\}.
	\]
	By Lemma~\ref{lem:restriction-strata}, for any point in the summand indexed
	by \(X\), among the equations from \(\A_{\bar\Gamma(1)}\) we have the exact
	condition
	\[
	x_1-x_j=i
	\quad\Longleftrightarrow\quad
	(i,j)\in A_X
	\qquad ((i,j)\in A(\bar\Gamma(1))).
	\]
	
	We next determine the contributing \(A_X\)'s.  If \(A_X\) contains two arcs
	with the same tail, say \((i,j)\) and \((i,k)\) with \(j\neq k\), then the two
	equations place \(j\) and \(k\) in the same box labeled by \(x_1-i\),
	contradicting the braid arrangement.  If \(A_X\) contains two arcs with the
	same head, say \((i,j)\) and \((i',j)\) with \(i\neq i'\), then \(j\) is forced into two
	different boxes, since \(1\le i,i'\le n<p\).  Hence every vertex has indegree and outdegree at most one in
	\(([n],A_X)\).  Thus \(([n],A_X)\) is a disjoint union of directed paths and
	directed cycles, with isolated vertices allowed.  Since no arc in the ambient
	indexing set ends at \(1\), the component containing \(1\) is a directed path.
	
	Conversely, let \(S\subseteq A(\bar\Gamma(1))\) be a spanning directed
	path-cycle cover, and put
	\[
	X=\bigcap_{(i,j)\in S}\{x_1-x_j=i\}.
	\]
	Then \(A_X=S\).  Indeed, \(S\subseteq A_X\) is immediate.  Conversely,
	suppose \((i,j)\in A_X\).  If no arc of \(S\) has head \(j\), then \(x_j\) is
	free in the defining equations of \(X\), so \(X\not\subseteq\{x_1-x_j=i\}\).
	Hence \(S\) has an arc with head \(j\).  Since \(S\) is a path-cycle cover,
	this arc is unique, say \((i_0,j)\).  Then \(X\) forces
	\(x_1-x_j=i_0\); as \(1\le i,i_0\le n<p\), the containment
	\(X\subseteq\{x_1-x_j=i\}\) gives \(i=i_0\).  Thus \((i,j)\in S\), and
	\(A_X=S\).  Hence the summands in \eqref{condition} to be counted are exactly those indexed by
	directed path-cycle covers of \(\bar\Gamma(1)\).
	
	Having identified the contributing \(X\)'s, fix one for which
	\(([n],A_X)\) has \(k\) directed path components.  Choose the box occupied
	by \(1\), giving \(p\) choices.  For each arc \((i,j)\in A_X\), the vertex
	\(j\) must occupy the box \(i\) steps counterclockwise from the box occupied
	by \(1\).  Since \(A_X\) has at most one arc with any given head, no vertex is
	prescribed twice; since it has at most one arc with any given tail and the
	boxes \(x_1-i\) for \(1\le i\le n\) are distinct modulo \(p\), no two
	prescribed vertices occupy the same box.  Hence the only vertices not yet
	placed are the initial vertices of directed path components; the one in the
	component containing \(1\) is already placed, so exactly \(k-1\) vertices
	remain.
	
	Let \(j\) be one of these remaining vertices.  For each \(1\le i\le n\), the
	box \(i\) steps counterclockwise from the box occupied by \(1\) is forbidden:
	if \((i,j)\in A(\Gamma)\), occupying it is forbidden by \(\Ish(\Gamma)\), while
	if \((i,j)\in A(\bar\Gamma(1))\), then \((i,j)\notin A_X\) because \(j\) has
	indegree zero, so occupying it is forbidden by the exact condition for
	\(\A_{\bar\Gamma(1)}\).  The box occupied by \(1\) is also forbidden by the
	braid arrangement.  Since \(p>2n\), these \(n+1\) forbidden boxes are distinct.
	Thus the \(k-1\) remaining vertices are placed injectively in
	\(p-n-1\) available boxes.  This fixed \(X\) contributes
	\[
	p(p-n-1)^{\underline{k-1}}.
	\]
	Summing over all directed path-cycle covers of \(\bar\Gamma(1)\), we obtain
	\[
	\#M_p(\Ish(\Gamma))
	=
	p\sum_{k=1}^{n}
	\pc_k(\bar\Gamma(1))
	(p-n-1)^{\underline{k-1}}.
	\]
	By the finite field method, the same formula holds for
	\(\chi(\Ish(\Gamma),p)\) for all sufficiently large primes \(p\).  Hence it
	holds as a polynomial identity after replacing \(p\) by \(t\).
\end{proof}

This formula is compatible with the Shi--Ish formulas of Armstrong and
Rhoades~\cite{armstrong2012}.  If
\(A(\Gamma_{\mathrm{Ish}})=\{(i,j):1\le i<j\le n\}\), then
\(\Ish(\Gamma_{\mathrm{Ish}})\) is the classical Ish arrangement, and
Theorem~\ref{thm:directed-ish-indexed} specializes to the characteristic polynomial
\(t(t-n)^{n-1}\).  More generally, when \(A(\Gamma)\) is the increasing
orientation of a simple graph \(G\), \(\Ish(\Gamma)\) is the corresponding
\(G\)-Ish arrangement; the formula above expresses its characteristic
polynomial through directed path-cycle covers of the indexing digraph
\(\bar\Gamma(1)\).

\subsection{Fixed-index Ish arrangements}\label{sec-fixed-first-index-ish}

This final family is independent of the directed \(G\)-Ish extension above.  For
a fixed \(r\in[n]\), the pair \((a,j)\in[n]\times([n]\setminus\{r\})\) indexes
the deleted hyperplane \(x_r-x_j=a\).  Thus loops \((a,a)\) are allowed when
\(a\ne r\), but no arc has head \(r\).

	Let \(n\) be a positive integer, let \(r\in[n]\), and let
	\(\Gamma=([n],A(\Gamma))\) be a digraph with
	\(A(\Gamma)\subseteq[n]\times([n]\setminus\{r\})\).  Define
	\[
	\mathcal J_r(\Gamma)
	=
	\mathcal C_n^n
	\setminus
	\{x_r-x_j=a:(a,j)\in A(\Gamma)\}.
	\]

For such a digraph \(\Gamma=([n],A(\Gamma))\) and \(a\in[n]\), let
\[
\deg^+_{\Gamma}(a)
=
\#\{j\in[n]\setminus\{r\}:(a,j)\in A(\Gamma)\}
\]
be the outdegree of \(a\) in \(\Gamma\).

For \(n=1\), the arrangement \(\mathcal J_1(\Gamma)\) is empty and has
characteristic polynomial \(t\).  The nontrivial case is \(n\ge2\).

\begin{theorem}\label{thm:fixed-first-index-ish}
	Let \(n\ge2\), let \(r\in[n]\), and let
	\(\Gamma=([n],A(\Gamma))\) be a digraph with
	\(A(\Gamma)\subseteq[n]\times([n]\setminus\{r\})\).  Then
	\[
	\chi(\mathcal J_r(\Gamma),t)
	=
	t(t-n^2-1)^{\underline{n-1}}
	+
	t\sum_{a=1}^{n}
	\deg^+_{\Gamma}(a)
	(t-n^2+n-a-1)^{\underline{n-2}}.
	\]
\end{theorem}

\begin{proof}
	Let \(p\) be a sufficiently large prime, chosen in particular so that
	\(p>n(n+1)\).  By the finite field method, it is
	enough to compute \(\#M_p(\mathcal J_r(\Gamma))\).  For a point in this
	complement, let
	\[
	S=\{(a,j)\in A(\Gamma):x_r-x_j=a\}
	\]
	be the set of deleted equations that hold.
	
	We first determine the possible \(S\)'s.  Suppose \(S\) contains two distinct
	arcs \((a,j)\) and \((a',j')\).  If \(j=j'\), then \(a\ne a'\), and the same vertex is forced
	to occupy two different boxes since \(1\le a,a'\le n<p\).  If \(j\ne j'\), then
	\(x_j-x_{j'}=a'-a\).  For \(a=a'\) this violates the braid arrangement; for
	\(a\ne a'\), since \(|a'-a|\le n-1<p\), one of \(x_j-x_{j'}\) and \(x_{j'}-x_j\) lies in
	\(\{1,\ldots,n-1\}\).  Since \(\mathcal J_r(\Gamma)\) deletes only
	hyperplanes of the form \(x_r-x_h=b\), and since \(j,j'\ne r\), this Catalan
	hyperplane remains in \(\mathcal J_r(\Gamma)\).  Hence \(S\) is either empty
	or a single arc.
	
	If \(S=\emptyset\), then any two consecutive occupied boxes in cyclic order
	must have at least \(n\) empty boxes between them.  Thus the placement
	consists of \(n\) singleton blocks, each followed by \(n\) empty boxes.
	Lemma~\ref{lem:cyclic-gap-count},
	with \(k=n\) and \(N=n(n+1)<p\), gives
	\[
	p(p-n^2-1)^{\underline{n-1}}
	\]
	points.
	
	Now let \(S=\{(a,j)\}\).  The vertices \(r\) and \(j\) form one rigid block:
	once the box occupied by \(r\) is chosen, \(j\) occupies the box \(a\) steps
	counterclockwise from it.  Thus this block has length \(a+1\).  Its \(a-1\)
	intermediate boxes must remain empty: an additional vertex there would either
	meet a retained Catalan hyperplane or satisfy a second deleted equation,
	contrary to the definition of \(S\).  The remaining
	\(n-2\) vertices form singleton blocks.  These \(n-1\) blocks must be
	separated cyclically by \(n\) empty boxes.  Therefore the total mandatory
	length is
	\[
	N=(a+1)+(n-2)+n(n-1)=n^2+a-1.
	\]
	Since \(a\le n\), we have \(N<n(n+1)<p\).
	There are \(n-1\) blocks, so Lemma~\ref{lem:cyclic-gap-count} gives
	\[
	p(p-N+n-2)^{\underline{n-2}}
	=
	p(p-n^2+n-a-1)^{\underline{n-2}}.
	\]
	For fixed \(a\), there are \(\deg^+_{\Gamma}(a)\) choices of \(j\).  Summing
	over \(a\in[n]\) and adding the contribution from \(S=\emptyset\), we get
	\[
	p(p-n^2-1)^{\underline{n-1}}
	+
	p\sum_{a=1}^{n}
	\deg^+_{\Gamma}(a)
	(p-n^2+n-a-1)^{\underline{n-2}}.
	\]
	Replacing \(p\) by \(t\) gives the stated characteristic polynomial.
\end{proof}

\phantomsection
\section*{Acknowledgements}

This work is supported by the National Natural Science Foundation of China
(Grant No.~12571350) and the Guangdong Basic and Applied Basic Research
Foundation (Grant No.~2025A1515010457).

\section*{AI Use Statement}

During the preparation of this manuscript, the authors used OpenAI’s ChatGPT and Codex to assist with language editing, computational checks, and the organization of the exposition. All mathematical results, proofs, and conclusions presented as contributions in this manuscript are the authors’ original work. The authors reviewed and revised all AI-assisted output and take full responsibility for the content of the manuscript.

\end{document}